\documentclass[letterpaper, 11pt,  reqno]{amsart}

\usepackage[margin=1.1in,marginparwidth=1.5cm, marginparsep=0.5cm]{geometry}

\usepackage{amsmath,amssymb,amscd,amsthm,amsxtra, esint}

\usepackage{mathrsfs} 

\usepackage[implicit=true]{hyperref}

\usepackage{color} 
\usepackage{ulem}
\usepackage[makeroom]{cancel}

\allowdisplaybreaks[2]

\definecolor{gr}{rgb}   {0.,   0.69,   0.23 }
\definecolor{bl}{rgb}   {0.,   0.5,   1. }
\definecolor{mg}{rgb}   {0.85,  0.,    0.85}
\definecolor{yl}{rgb}   {0.8,  0.7,   0.}
\definecolor{or}{rgb}  {0.7,0.2,0.2}

\newtheorem{theorem}{Theorem} [section]

\newtheorem{lemma}[theorem]{Lemma}
\newtheorem{proposition}[theorem]{Proposition}
\newtheorem{remark}[theorem]{Remark}

\newcommand{\noi}{\noindent}
\newcommand{\Z}{\mathbb{Z}}
\newcommand{\R}{\mathbb{R}}

\newcommand{\T}{\mathbb{T}}

\let\P= \undefined
\newcommand{\P}{\mathbf{P}}

\newcommand{\E}{\mathbb{E}}

\newcommand{\be}{\beta}
\newcommand{\dl}{\delta}

\newcommand{\eps}{\varepsilon}

\newcommand{\G}{\Gamma}
\newcommand{\ld}{\lambda}

\newcommand{\s}{\sigma}

\newcommand{\ft}{\widehat}

\newcommand{\wt}{\widetilde}
\newcommand{\cj}{\overline}
\newcommand{\dx}{\partial_x}
\newcommand{\dt}{\partial_t}

\newcommand{\embeds}{\hra}

\renewcommand{\l}{\ell}
\renewcommand{\o}{\omega}

\newcommand{\les}{\lesssim}
\newcommand{\ges}{\gtrsim}

\newcommand{\jb}[1]
{\langle #1 \rangle}

\newcommand{\ind}{\mathbf 1}

\newcommand{\N}{\mathbb{N}}

\renewcommand{\H}{\mathcal{H}}

\newtheorem*{ackno}{Acknowledgements}

\numberwithin{equation}{section}
\numberwithin{theorem}{section}

\DeclareMathOperator{\sgn}{sgn}

\newcommand{\Q}{\mathcal{Q}}

\newcommand{\hra}{\hookrightarrow}

\makeatletter
\@namedef{subjclassname@2020}{%
  \textup{2020} Mathematics Subject Classification}
\makeatother

\begin{document}
\baselineskip = 14pt

\title[QI for BO-BBM]
{Improved quasi-invariance result for the periodic Benjamin-Ono-BBM equation
}

\author[J.~Forlano]
{Justin Forlano}

\address{
Justin Forlano\\ 
School of Mathematics\\
Monash University\\
 VIC 3800\\
 Australia\\
 and\\
School of Mathematics\\
The University of Edinburgh\\
and The Maxwell Institute for the Mathematical Sciences\\
James Clerk Maxwell Building\\
The King's Buildings\\
Peter Guthrie Tait Road\\
Edinburgh\\ 
EH9 3FD\\
United Kingdom}

\email{justin.forlano@monash.edu}



\begin{abstract}
We extend recent results of Genovese-Luca-Tzvetkov (2022) regarding the quasi-invariance of Gaussian measures under the flow of the periodic Benjamin-Ono-BBM (BO-BBM) equation to the full range where BO-BBM is globally well-posed. The main difficulty is due to the critical nature of the dispersion which we overcome by combining the approach of Coe-Tolomeo (2024) with an iteration argument due to Forlano-Tolomeo (2024) to obtain long-time higher integrability bounds on the transported density.
\end{abstract}



%
\maketitle


\vspace*{-6mm}

\section{Introduction}

We consider the Benjamin-Ono-BBM (BO-BBM) equation:
\begin{equation}
\begin{cases}\label{BOBBM}
\dt u +|D_x|\dt u+\dx u + \dx(u^2)=0 \\
u|_{t = 0} = u_0,
\end{cases}
\qquad ( t, x) \in \R \times \T,
\end{equation}
where $|D_x|$ is the Fourier multiplier operator with multiplier $|n|$. The BO-BBM equation \eqref{BOBBM} arises as a model for the evolution of long-crested waves at the interface between two immiscible fluids.
It is closely connected to the Benjamin-Ono (BO) equation \cite{Benjamin, Ono}:
\begin{align}
\dt u +|D_x|\dx u+\dx u + \dx(u^2)=0,
\label{BO}
\end{align}
as, formally, \eqref{BOBBM} arises from replacing $\dx$ in the second term in \eqref{BO} by $\dt$. 
This substitution has some physical significance related to the 
formal derivation of these models, see \cite{BBM, BK}.

To proceed with an analytical point of view we may rewrite \eqref{BOBBM} as 
\begin{align}
\dt u =-(1+|D_x|)^{-1}\dx ( u+ u^2). \label{BOBBM2}
\end{align}
As compared to \eqref{BO}, the derivative in the nonlinearity of \eqref{BOBBM} has been completely ameliorated. In this sense \eqref{BOBBM} is also referred to as the regularised Benjamin-Ono equation  \cite{AB}. This moniker is partly merited by the following fact: whilst the well-posedness problem for the BO equation has historically been a difficult topic culminating in \cite{GKT, KLV} and the references therein, the well-posedness problem for BO-BBM \eqref{BOBBM} is much simpler.

To describe the well-posedness situation further, we need to set some notation.
We define the Fourier transform by
\begin{align*}
\ft u (n)  =\frac{1}{2\pi} \int_{0}^{2\pi} e^{-inx} u(x) dx.
\end{align*}
and so that $ \mathcal{F}\{ \dx u\}(n) = in \ft u(n)$ and 
\begin{align*}
\|u\|_{L^2(\T)}^{2} = 2\pi  \sum_{n\in \Z_{\ast}} |\ft u(n)|^2.
\end{align*}
The mean $\int_{\T}udx$ is a conserved quantity for \eqref{BOBBM} and this allows us to, although not  necessarily \cite{ASB}, restrict our discussions to the mean zero Sobolev spaces $H^{\s}_{0}(\T)$, with the norm
\begin{align*}
\|u\|_{H^{\s}}^{2} = \sum_{n\in \Z_{\ast}} |n|^{2\s}|\ft u(n)|^2,
\end{align*}
where $\Z_{\ast}:= \Z\setminus \{0\}$.  
As $H^{\s}(\T)$ is an algebra when $\s>1/2$ and the operator $(1+|D_x|)^{-1}\dx$ is bounded on $H^{\s}(\T)$, we obtain local well-posedness for \eqref{BOBBM2} in $H^{\s}(\T)$. 
Global well-posedness for $\s \geq 1/2$ is also known, with the case $\s>1/2$ in \cite{Mammeri1, ASB} following by combining the Brezis-Gallou\"{e}t inequality with conservation of the energy 
\begin{align}
E(u)  = \big(\|u\|_{L^2}^{2}+ 2\pi  \|u\|_{H^{1/2}}^{2} \big)^{1/2}. \label{E}
\end{align}
The global well-posedness at the endpoint $\s =1/2$ was shown in \cite{Mammeri2} by a compactness argument from conservation of $E$ and a Yudovich-type argument to recover uniqueness. Strictly speaking, \cite{Mammeri2} established this result on the real line $\R$ but the proof extends without additional complication to the periodic setting.
Thus, we have a well-defined global-in-time flow $u_0 \mapsto \Phi_{t}(u_0)$ on $H^{\s}(\T)$ for any $\s \geq 1/2$. 
The well-posedness/ill-posedness of \eqref{BOBBM2} is currently open for $\s<1/2$. See \cite{ASB} for a mild form of ill-posedness (failure of $C^2$-regularity for the solution map) when $\s<0$.


Our interest in \eqref{BOBBM} is to improve upon the known
regularity ranges under which Gaussian measures remain absolutely continuous under the flow of \eqref{BOBBM} and to the point at which agrees with the well-posedness theory.
 In the setting of Hamiltonian PDEs, this direction was initiated by Tzvetkov~\cite{Tz} and has since seen a remarkable amount of progress \cite{Tz, OTz, OTz2, OST, OTT, PTV, FT, GOTW, STX, DT2020,  GLT1, GLT2, PTV2, BTh, FS, OS, FT2, ST, CT, Knezevitch, FT3}. In particular, all of these results together show that there is a delicate balance between positive results in this area (quasi-invariance) and the amount of dispersion and/or structure in the given equation. 
In this vein, BO-BBM \eqref{BOBBM} is an interesting model for this study as the amount of dispersion seems to be borderline for positive quasi-invariance results to hold; see \cite{GLT2}, Theorem~\ref{THM:QI}, and the ensuing discussion below.

 Let us first describe this quasi-invariance problem in more detail, focusing on the situation for the mean-zero flow of BO-BBM \eqref{BOBBM}.
Given $s\in \R$, we consider the Gaussian measures $\mu_s$ formally prescribed as
\begin{align*}
d\mu_s = Z_{s}^{-1} e^{-\frac{1}{2}\|u\|_{H^{s+1/2}}^{2}} du,
\end{align*}
and rigorously defined as the law of the random Fourier series
\begin{align}
u_{0}^{\o} (x) = \sum_{n\in \Z_{\ast}} \frac{g_n (\o)}{|n|^{s+1/2}} e^{inx}, \label{u0}
\end{align}
where $\{g_{n}(\o)\}_{n\in \Z_{\ast}}$ are a family of complex-valued standard normal random variables on some probability space $(\Omega, \mathcal{F}, \mathbb{P})$ and conditioned so that $g_{-n}(\o) = \cj{g_n(\o)}$. 
These Gaussian measures are supported on the Sobolev spaces $H^{\s}_{0}(\T)$ for any $\s<s$.
Previously, Genovese-Luca-Tzvetkov~\cite{GLT2} have shown that for every $s>1$,
\begin{align}
(\Phi_{t})_{\#}\mu_{s} \sim \mu_s \quad \text{for all} \quad t\in \R, \label{QI}
\end{align} 
where $\sim$ denotes the mutual absolute continuity of measures, namely, $(\Phi_{t})_{\#}\mu_{s}  \ll \mu_s$ and $\mu_s \ll (\Phi_{t})_{\#}\mu_{s}$. The property \eqref{QI} is called quasi-invariance.

The authors in \cite{GLT2} were able to say more than \eqref{QI}, as we describe now. 
Given $R>0$, 
 we denote by $B_{R}$ the closed ball centered at zero of radius $R$ in the energy space; namely, 
\begin{align}
B_{R} : =\big \{ u\in H_0^{1/2}(\T)\,:\, E(u)\leq R\big\},
\end{align}
and define the localised measure
\begin{align}
d\mu_{s,R} = \ind_{B_R}(u) d\mu_s. \label{musR}
\end{align}
The addition of the (invariant) cut-off $\ind_{B_R}$ allowed them to show that the transported density (Radon-Nikodym derivative) 
\begin{align}
f_{t,R} : = \frac{d(\Phi_{t})_{\ast}\mu_{s,R}}{d\mu_{s,R}} \label{density}
\end{align}
enjoys higher $L^p$-integrability, than merely belonging to $L^{1}(d\mu_{s,R})$.
However, as a consequence of the critical nature of the dispersion in \eqref{BOBBM2}, the integrability index $p=p(t,R)$ depends on $t>0$ and has the property that $p(|t|,R)\to 1$ as $|t|\to +\infty$. 
Interestingly, this deterioration of the exponent does not occur for models with even slightly stronger dispersion, such as for the fractional BBM equations 
\begin{align}
\dt u +|D_x|^{\be}\dt u+\dx u + \dx(u^2)=0  
\end{align}
with $\be>1$ (where $\be=1$ corresponds to \eqref{BOBBM}); see \cite{GLT1}. It is in this sense that we see BO-BBM \eqref{BOBBM} as a model with a critical amount of dispersion for quasi-invariance to hold. 
Note that a critical aspect also appears for nonlinear wave equations \cite{OTz2, GOTW}.



Our goal is to extend the results of \cite{GLT2} to the full range $s>1/2$ where BO-BBM~\eqref{BOBBM} is well-posed. 


\begin{theorem}\label{THM:QI}
Given $s>\frac12$ and $t\in \R$, the transported measure $(\Phi_{t})_{\ast}\mu_{s}$ is mutually absolutely continuous with respect to the Gaussian measure $\mu_s$. Moreover, for each $t>0$ and $R>0$ there exists $p=p(|t|,R)>1$ such that $p(|t|,R)\to 1$ as $|t|\to \infty$, and for which the Radon-Nikodym derivative $f_{t,R}$ defined in \eqref{density}
satisfies
\begin{align}
 \|f_{t,R}\|_{L^{p}(d\mu_{s,R})} \leq  C(R,p). \label{ftbds}
\end{align}
\end{theorem}

The regularity restriction of $s>1$ in \cite{GLT2} arises from an energy estimate on essentially $\log f_{t,R}$, which is given in \eqref{QsN}. This requires $\mu_s$ almost sure finiteness of $\| \dx u\|_{L^{\infty}}$; see \eqref{Qeq}.
Our strategy to go beyond the analysis in \cite{GLT2} is to adapt the recent approach of \cite{CT} which allows to use probability in a more complete way as well as make use of a symmetrisation argument in \eqref{Q1}.
The main issues is the exponential integrability of essentially $\log f_{t,R}$ (Lemma \ref{LEM:expint})
and the variational formula \cite{BD, Ust, BG} provides a sharp way to prove this, where we can leverage the random oscillations coming from multilinear expressions of the random initial data \eqref{u0}. 

However, due to the critical dispersion of BO-BBM \eqref{BOBBM} it is not enough to simply apply the method in \cite{CT}. Indeed, we are not able to obtain \eqref{Lpbd} as directly as one would in \cite{CT} because of the short-time restriction in Lemma~\ref{LEM:LPbound}.
We need to go one step beyond this by carefully iterating these bounds on short time intervals to reach long times (Proposition~\ref{PROP:LPboundlong}). For this step, we adapt an argument from \cite{FT3}.
At a high-level, we advocate that this strategy arising from combining ideas from \cite{CT, FT3} provides an effective way to analyse the low-regularity quasi-invariance problem for equations with a borderline amount of dispersion.

As discussed in \cite{GLT2}, a complete solution to the program initiated in \cite{Tz} would constitute finding a formula for the transported density with respect to the unrestricted Gaussian measure $\mu_s$. This has only been achieved in a few works \cite{GLT1, FS, PTV2, Knezevitch} and relies on more dispersion. It seems to be an  interesting but difficult problem to determine if (i) one could remove the localisation $B_{R}$ in \eqref{ftbds}, and (ii) obtain a formula for $f_{t,R}$, and the density with respect to $\mu_s$.
Indeed, there is no multilinear smoothing in the nonlinear term in \eqref{BOBBM2}, even under randomisation: if $u$ were replaced by $u_0^{\o}$ given in \eqref{u0}. 


The analysis in this paper simplifies dramatically if one
 is only interested in the qualitative statement of quasi-invariance as in \eqref{QI} and not the higher integrability property of the transported density. Indeed, in this case it is possible to instead argue solely as in \cite{CT} by replacing the cut-off in \eqref{musR} by a cut-off to a ball in $H^{\s}_{0}(\T)$ for $\frac 12<\s<s$. Namely, there is no need to use an argument as in Proposition~\ref{PROP:LPboundlong} as there is no longer an upper bound restricition on the parameter $\ld$ in Lemma~\ref{LEM:expint}. 
 However, such an approach does not say anything about higher integrability of the density, and thus is not in the spirit of identifying a formula for the density. 

The aforementioned lack of nonlinear smoothing under randomisation also complicates efforts to extend Theorem~\ref{THM:QI} to some range $s\leq 1/2$. 
The first step here would be to construct local in time solutions almost surely with respect to $\mu_s$, but a first-order (Da-Prato Debussche) expansion of the solution does not work due to the presence of high-low interactions. Thus, even almost sure \textit{local} well-posedness on the support of the Gaussian measure $\mu_s$ is non-trivial. Even in the presence of an almost surely globally well-defined flow with respect to $\mu_s$, we conjecture that quasi-invariance fails when $0<s\leq \frac 12$ and that the transported measure should be singular with respect to $\mu_s$. 
To give some evidence for this, we recall that Coe-Tolomeo~\cite{CT} provided a benchmark test as to when quasi-invariance may or may not hold. Their criterion, in our setting, amounts to computing the quantity
\begin{align*}
\text{QI}_{s,N}(t):=\E \bigg[ \bigg| \int_{0}^{t} Q_{s,N}( S(-t') u_0^{\o})dt'\bigg|^{2}\bigg],
\end{align*}
where $Q_{s,N}$ is defined in \eqref{QsN} and $S(t):=e^{t(1+|D_x|)^{-1}\dx}u_0^{\o}$ is the linear solution for \eqref{BOBBM2} with the Gaussian random initial data \eqref{u0}. If $\text{QI}_{s,N}$ is finite uniformly in $N$, we may expect quasi-invariance, while if it is not, we may expect singularity. In Section~\ref{SEC:sing}, we prove the following:
\begin{proposition}\label{PROP:s12}
Let $0< s \leq \frac 12$.
For all $N$ sufficiently large and $|t|\les 1$ it holds that
\begin{align}
\textup{QI}_{s,N}(t) \ges 
\begin{cases}
 t^{2} N^{1-2s} \quad & \text{if} \quad s<\frac 12,\\
 t ^{2}(\log N)^2 \quad & \text{if} \quad  s=\frac 12.
\end{cases}
\label{QIfail}
\end{align}
\end{proposition}
\noi
Note that we omitted the case $s=0$ in Proposition~\ref{PROP:s12} as formally we expect  the measure $Z^{-1} e^{-E(u)^2}du$, where $E$ is as in \eqref{E}, to be invariant under the flow of \eqref{BOBBM2}. Of course, it is unclear if there is any well-defined dynamics almost surely on the support of this measure.


\section{Preliminaries}

For $x\in \R$, we define $\jb{x}= \sqrt{1+x^2}$. 
We let $\N=\{1,2,\ldots\}$ and $\N_0 =\N\cup \{0\}$. Given $n_1,n_2,n_3\in \Z$, we let $n_{123}:=n_1+ n_2+n_3$, and, for $j=1,2,3$, denote by $n_{(j)}$ the $j$th largest of these numbers in magnitude so that $|n_{(1)}|\geq |n_{(2)}|\geq |n_{(3)}|$. Given $1\leq p\leq \infty$, we let $1\leq p'\leq \infty$ denote its H\"{o}lder conjugate satisfying $\frac{1}{p}+\frac{1}{p'}=1$.

We define Littlewood-Payley operators as follows:
Let $\eta \in C^{\infty}_{c}(\R)$ be a smooth bump function such that $0\leq \eta\leq 1$, $\eta(\xi)=1$ for $|\xi|\leq 1$ and $\eta(\xi)=0$ for $|\xi|\geq 2$.  We define $\phi(\xi):=\eta(\xi)-\eta(2\xi)$ and for $N\in 2^{\Z}$, we set $\phi_{N}(\xi)=\phi(\xi/N)$. For $N\in 2^{\Z}$, we define projectors 
\begin{align*}
\ft{\P_N f}(n) = \phi_{N}(n)\ft f(n), \quad \ft{\P_{\leq N} f}(n) = \eta(\tfrac{n}{N})\ft f(n), \quad \ft{ \P_{>N} f}(n)=(1-\eta(\tfrac{n}{N})) \ft f(n).
\end{align*}
We define the wider projector $\wt{\P}_{N}:= \P_{\leq 8N}-\P_{\leq N/8}$, which has the property that $\wt{\P}_{N} \P_{N} = \P_{N}$.

Given $N\in \N$, we define the projected Gaussian measures $\mu_{s,N}:= (\P_{\leq N})_{\#}\mu_s$ which are Gaussian measures in $\R^{2N}$, and we define $\mu_{s,N}^{\perp}$ as the law of the random Fourier series
\begin{align*}
u_{N}^{\perp} (x;\o) = \sum_{|n|>N} \frac{g_{n}(\o)}{ |n|^{s+1/2}}e^{inx}.
\end{align*}

\begin{lemma}
Let $s>\frac 12$ and $N\in \N$. Then, for any $M\geq 1$, it holds that 
\begin{align}
\mu_{s,N}^{\perp} \big( \| u_{N}^{\perp}\|_{H^{1/2}_0(\T)}^2 >M\big) \les M^{-2}N^{-2s+1}
\label{decayest}
\end{align}
where the implicit constant is uniform in $M$ and $N$.
\end{lemma}
\begin{proof}
We have 
\begin{align*}
\E [\| u_{N}^{\perp}\|_{H^{1/2}_0(\T)}^{2} ]&   \les \E \bigg[ \sum_{|n|>N} \frac{|g_n|^2}{|n|^{2s}}\bigg]
 \les  \sum_{|n|>N} \frac{1}{|n|^{2s}}   \les N^{-2s+1},
\end{align*}
at which point, \eqref{decayest} follows from Markov's inequality.
\end{proof}

\noi
We can obtain an exponential decay rate in \eqref{decayest} but the polynomial one here will be sufficient.

\section{Localised density bounds}

Given $N\in \N$, we define the truncated version of \eqref{BOBBM2}:
\begin{align}
 \dt u_{N} = -(1+|D_x|)^{-1}\dx \P_{\leq N} ( u_N + u_N^2), \quad u_{N}|_{t=0} = \P_{\leq N} u_0 \label{BON}
\end{align}
Similar to \eqref{BOBBM}, for every fixed $N\in \N$, \eqref{BON} is globally well-posed in $H^{\s}$, for $\s>\frac 12$, with bounds on the Sobolev norms of the solutions that are uniform in $N$. This can be proved by repeating the arguments in \cite{Mammeri1} and using the $L^2$-boundedness of $\P_{\leq N}$.
 We denote the flow of \eqref{BON} by $\Phi^{N}_{t}$. Note that if $u_N$ solves \eqref{BON}, then $u_N\in C^{\infty}(\R\times \T)$, which justifies the ensuing computations.

As the truncated equation \eqref{BON} is a finite dimensional Hamiltonian system (in the Fourier coefficients for $u_N$), Liouville's theorem implies that the Lebesgue measure $L_N$ on $\R^{2N}$ is invariant under $\Phi_{t}^{N}$. Thus, for every $t\in \R$,
\begin{align}
d(\Phi^{N}_{t})_{\#} \mu_{s,N}  &= f_{t}^{N}(u) d\mu_{s,N} := \exp\Big(  - \tfrac{1}{2}\| \Phi^{N}_{t}(u)\|_{H^{s+1/2}}^{2} + \tfrac{1}{2} \|u\|_{H^{s+1/2}}^{2} \Big)d\mu_{s,N}, \label{density}
\end{align}
where $f_{t}^{N}$ is the transported density. 
By the fundamental theorem of calculus, we have 
\begin{align*}
 -\tfrac{1}{2}\| \Phi^{N}_{t}(u)\|_{H^{s+1/2}}^{2} + \tfrac{1}{2} \|u\|_{H^{s+1/2}}^{2} = \int_{0}^{-t} \mathcal{Q}_{s,N}(\Phi_{t'}^{N}(u))dt',
\end{align*}
where $\Q_{s,N}(u) : =\Q_{s,N}(u,u,u)$ is the trilinear operator 
\begin{align}
\mathcal{Q}_{s,N}(u) &: = - \frac{d}{dt} \bigg(\frac 12 \| \Phi_{t}^{N}(u)\|_{H^{s+1/2}}^2   \bigg)  \Bigg\vert_{t=0}. \label{QsN}
\end{align}
From \eqref{BON}, the identity $(1+|D_x|)^{-1} |D_x| = 1-(1+|D_x|)^{-1}$
and for $u_{N}:=\P_{\leq N} u$, we then have
\begin{align}
\begin{split}
\mathcal{Q}_{s,N}(u) & = \jb{|D_x|^{2s+1} u_{N}, (1+|D_x|)^{-1} \dx u_{N}} +  \jb{|D_x|^{2s+1} u_N, (1+|D_x|)^{-1} \dx (u_N^2)} \\
&  =   \jb{|D_x|^{2s} u_N, \dx (u_N^2)} -  \jb{|D_x|^{2s} u_N, (1+|D_x|)^{-1} \dx (u_N^2)}.
\end{split} \label{Qeq}
\end{align}
Notice that the contribution from the linear part in \eqref{Qeq} vanishes identically since
\begin{align*}
\jb{|D_x|^{2s+1} u_{N}, (1+|D_x|)^{-1} \dx u_{N}} = \jb{ |D_x|^{s+\frac 12}(1+|D_x|)^{-\frac 12}u_N, \dx |D_x|^{s+\frac 12}(1+|D_x|)^{-\frac 12}u_N   }=0.
\end{align*}
By a symmetrisation argument, we write \eqref{Qeq} this more favourably as
\begin{align}
\mathcal{Q}_{s,N}(u) = \Q^{(1)}_{s,N}(u) +  \Q^{(2)}_{s,N}(u), \label{Qdecomp}
\end{align}
where we define the trilinear operators $\Q^{(j)}_{s,N}(u):= \Q^{(j)}_{s,N}(u,u,u)$, $j=1,2$, for $u\in \Pi_{\leq N}L^2$, by 
\begin{align}
\Q^{(1)}_{s,N}(u_1,u_2,u_3)  & =  \sum_{\substack{ n_1 +n_2+n_3=0 \\ 0<|n_j|\leq N, j=1,2,3} } i\Psi_{s}(\cj{n}) \ft u_1 (n_1)  \ft u_2(n_2) \ft u_3(n_3)  \label{Q1}  \\
 \Q^{(2)}_{s,N}(u_1,u_2,u_3)   &= \jb{|D_x|^{2s}u_1, (1+|D_x|)^{-1}\dx (u_2 u_3)} \label{Q2}
\end{align}
and 
\begin{align}
\Psi_{s}(\cj{n}):= \Psi_{s}(n_1,n_2,n_3): = |n_1|^{2s+1}\sgn(n_1)  + |n_2|^{2s+1}\sgn(n_2)+|n_3|^{2s+1}\sgn(n_3),
\end{align}
The mulitilinear operator $\Q^{(1)}_{s,N}$ is worse behaved as compared to $\Q^{(2)}_{s,N}$ because it lacks the factor $(1+|D_x|)^{-1}$. Nonetheless, the symmetrisation in $\Q^{(1)}_{s,N}$ allows us to always place the derivative $\dx$ onto the lowest frequency term, in the following sense.

\begin{lemma}  \label{LEM:Psis}
Let $s>\frac 12$ and $n_1,n_2,n_3\in \Z_{\ast}$ such that $n_1 +n_2+n_3=0$. Then, either \textup{(i)} $|n_1|\sim |n_2|\sim |n_3|$  or 
\textup{(ii)} $|n_{(2)}| \gg |n_{(3)}|$ and hence $|\Psi_{s}(\cj{n})| \les |n_{(1)}|^{2s} |n_{(3)}|.$
\end{lemma}
\begin{proof}
By symmetry, we may assume that $|n_1|\geq |n_2|\geq |n_3|$. Thus, when $|n_2|\gg |n_3|$, we must have $\sgn(n_1)=-\sgn(n_2)$, so by the mean value theorem,
$$
|\Psi_{s}(\cj{n})| \leq  ||n_1|^{2s+1} - |(-n_2)|^{2s+1}| +|n_3|^{2s+1}  \les |n_1|^{2s}|n_3| + |n_3|^{2s+1}\sim |n_1|^{2s}|n_3|.  
\eqno \qedhere
$$
\end{proof}

We note that since $E$ in \eqref{E} is a conserved quantity on the support of $\mu_s$, $B_{R}$ is an invariant set under $\Phi^{N}_{t}$; namely, $(\Phi^{N}_{t})^{-1}(B_R)=B_{R}$ for any $t\in \R$. 

\begin{remark}\rm \label{RMK:densityfinite}
We note that the density $f_{t}^N$ localised to $B_{R}$ always has finite $L^p(d\mu_s)$-moments, albeit with a severe loss in uniformity in $N$. Indeed, for $u\in B_R$, we have $\Phi_{t}^{N}(u)\in B_{R}$ and thus Sobolev embeddings imply
\begin{align*}
|\Q_{s,N}(\Phi_{t}^{N}(u))| \les N^{2s+\frac{1}{2}} \|\Phi_{t}^{N}(u)\|_{H^{\frac 12}}^{3} \les N^{2s+\frac{1}{2}}R^{\frac 32}.
\end{align*}
Our goal is to obtain bounds which are uniform in $N$.
\end{remark}

The main ingredient in this article is the following exponential integrability statement.

\begin{lemma}\label{LEM:expint}
Let $\frac 12 <s\leq 1$, $N\in \N$, and $R>0$. Then, there exists constants $c_0, A>0$ independent of both $R$ and $N$ such that for all $0\leq \ld \leq c_0 R^{-1}$, it holds that 
\begin{align*}
\sup_{N\in \N} \int \exp\big( \ld |\Q_{s,N}(u)| \ind_{B_{R}}(u)\big) d\mu_{s,N}  \leq e^{C(R) \ld^{A}}.
\end{align*}
\end{lemma}

\noi
We postpone the proof of Lemma~\ref{LEM:expint} to Section~\ref{SEC:expint}.
Using Lemma~\ref{LEM:expint}, we obtain the following uniform in $N$, short-time $L^p$ bounds on the densities. 

\begin{lemma}\label{LEM:LPbound}
Let $\frac 12 <s\leq 1$, $N\in \N$, and $R>0$. Then, there exists $c_0>0$ independent of both $R$ and $N$ such that for any $t\in \R$ and $1\leq p<\infty$ satisfiying  $0\leq p| t|\leq c_0 R^{-1}$, there exists $C(R,p)>0$ such that
\begin{align}
\sup_{N\in \N} \| f_{t}^{N}\ind_{B_{R}}\|_{L^p(d\mu_{s,N})} \leq e^{C(R)(|t|p)^A}. \label{Lpbd2}
\end{align}
\end{lemma}
\begin{proof}

The proof largely follows that of \cite[Proposition 3.6]{CT}, however, as $B_{R}$ is invariant under $\Phi^{N}_{t}$, there are some simplifications. For the convenience of the reader, we provide the details. 
First, we reduce the proof of the bound \eqref{Lpbd2} for $f_{t}^N \ind_{B_R}$ in such a way that it remains to apply Lemma~\ref{LEM:expint}. Let $I_{t}=[0,t]$ for $t\geq 0$ and $I_t= [t,0]$ for $t<0$.  Let $t\neq 0$.
By the formula \eqref{density}, the group property of the flow $\Phi^{N}_t$, Jensen's inequality, Fubini's theorem, and the invariance of $B_{R}$,
we have
\begin{align*}
\| f_{t}^{N} \ind_{B_R}\|_{L^p (\mu_{s,N})}^{p}  & = \int (f_{t}^{N}(u))^{p-1} \ind_{B_{R}}(u) f_{t}^{N}(u) d\mu_{s,N} \\
& = \int (f_{t}^{N}(\Phi_{t}^{N}(u)))^{p-1} \ind_{B_{R}}(u) d\mu_{s,N}(u) \\
&= \int \exp\bigg( -(p-1)\int_{0}^{-t} \Q_{s,N}( \Phi_{t+\tau}^{N}(u))d\tau\bigg) \ind_{B_{R}}(u)d\mu_{s,N}(u)\\
& = \int \exp\bigg( \frac{t}{|t|}(p-1)\int_{I_t}\Q_{s,N}( \Phi_{\tau}^{N}(u))d\tau\bigg) \ind_{B_R}(u)d\mu_{s,N}(u) \\
& \leq  \frac{1}{|t|} \int \int_{I_t} \exp\big( t(p-1) \Q_{s,N}(\Phi_{\tau}^{N}(u))\big) d\tau \ind_{B_R}(u)d\mu_{s,N}(u)  \\
& =  \frac{1}{|t|} \int_{I_t}  \int \exp\big( t(p-1) \Q_{s,N}(\Phi_{\tau}^{N}(u))\big)  \ind_{B_R}(u)d\mu_{s,N}(u) d\tau \\
& =  \frac{1}{|t|} \int_{I_t} \int \exp\big( t(p-1) \Q_{s,N}(u)\big) \ind_{B_R}(u) f_{\tau}^{N}(u)d\mu_{s,N}(u)d\tau \\
& \leq \big\| e^{|t(p-1)\Q_{s,N}(u)|}\ind_{B_R}\|_{L^{p'}(d\mu_{s,N})} \sup_{\tau \in I_{t}}\|f_{\tau}^{N}\ind_{B_R}\|_{L^{p}(d\mu_{s,N})} \\
& = \bigg( \int e^{|tp\Q_{s,N}(u)|}\ind_{B_R}(u)d\mu_{s,N}(u)\bigg)^{\frac{1}{p'}} \sup_{\tau \in I_{t}}\|f_{\tau}^{N}\ind_{B_R}\|_{L^{p}(d\mu_{s,N})}.
\end{align*}
Note that for each fixed $\tau \in I_t$, we can repeat this argument to also obtain
\begin{align*}
\sup_{\tau \in I_{t}}\| f_{\tau}^{N} \ind_{B_R}\|_{L^p (\mu_{s,N})}^{p} \leq \bigg( \int e^{|tp\Q_{s,N}(u)|}\ind_{B_R}(u)d\mu_{s,N}(u)\bigg)^{\frac{1}{p'}} \sup_{\tau \in I_{t}}\|f_{\tau}^{N}\ind_{B_R}\|_{L^{p}(d\mu_{s,N})},
\end{align*}
where we used that $e^{|tp\Q_{s,N}(u)|}$ is increasing in $|t|$. Thus, by Remark~\eqref{RMK:densityfinite}, we find
\begin{align*}
\sup_{\tau \in I_{t}}\| f_{t}^{N} \ind_{B_R}\|_{L^p (\mu_{s,N})} &\leq \bigg(\int e^{|tp\Q_{s,N}(u)|}\ind_{B_R}(u)d\mu_{s,N}(u) \bigg)^{\frac1p} \\
&=  \bigg(\int e^{|tp\Q_{s,N}(u)|\ind_{B_R}(u)}d\mu_{s,N}(u) -\mu_{s,N}(B_{R}) \bigg)^{\frac1p} \\
& \leq \bigg(\int e^{|tp\Q_{s,N}(u)|\ind_{B_R}(u)}d\mu_{s,N}(u) \bigg)^{\frac 1p}.
\end{align*}
The estimate \eqref{Lpbd2} now follows from this and Lemma~\ref{LEM:expint}.
\end{proof}

We now iterate the short-time bounds \eqref{Lpbd} to reach arbitrary times $t\in \R$; namely, to times which no longer depend on $R$. The argument here is similar to that in \cite{FT3}. We point out again that this extra step seems to be necessary due to the restriction on $\ld$ in Lemma~\ref{LEM:expint} which itself is a symptom of the critical nature of \eqref{BOBBM} and the use of an energy cut-off.

\begin{proposition}\label{PROP:LPboundlong}
Let $\frac 12 <s\leq 1$, $t\in \R$, and $R>0$. Then, there exists $p(t,R)>1$ and 
$C(t,R)>0$ such that
\begin{align}
\sup_{N\in \N} \| f_{t}^{N}\ind_{B_{R}}\|_{L^{p(t,R)}(d\mu_{s,N})} \leq C(t,R). \label{Lpbd}
\end{align}
\end{proposition}

The strategy is similar to that of \cite[Proposition 4.14]{FT3}, however, to make this article as self-contained as possible, we provide some of the details below as well as the modifications. In particular, to obtain $p(t,R)$ for general $t\in \R$, we need to more carefully choose the exponents later in the argument as compared to \cite{FT3}. 

\begin{proof}
First, we claim that 
\begin{align}
\| f_{t+\tau}^{N} \ind_{B_R}\|_{L^r(d\mu_{s,N})} \leq \| f_{\tau}^{N}\ind_{B_R}\|_{L^p(d\mu_{s,N})} \| f_{t}^{N}\ind_{B_R}\|_{L^{\frac{r(p-1)}{p-r}}(d\mu_{s,N})}^{\frac{1}{p'}} \label{timestep}
\end{align}
for any $t, \tau \in \R$ and $1\leq r<p<\infty$.
Let $F:H^{\frac 12}(\T)\to \R$ be continuous and bounded. Then, by the defining property of the push-froward measure $(\Phi_{t}^{N})_{\#} \mu_s$, the group property of the flows $\Phi_{t}^{N}$, the invariance of $B_{R}$ under the flow, and H\"{o}lder's inequality, we have
\begin{align}
\bigg|\int F(u) f_{t+\tau}^{N}(u) \ind_{B_R} d\mu_{s,N}(u)\bigg| &=\bigg| \int F(\Phi_{t+\tau}^{N}(u)) \ind_{B_R}(\Phi_{t+\tau}^{N}(u)) d\mu_{s,N}\bigg| \notag\\
& =\bigg| \int F( \Phi^{N}_{\tau}\circ \Phi_{t}^{N}(u)) \ind_{B_{R}}(\Phi^{N}_{\tau}\circ \Phi_{t}^{N}(u)) d\mu_{s,N}\bigg| \notag\\
& = \bigg|\int F(\Phi_{t}(u))\ind_{B_{R}}(u) f_{\tau}^{N}(u) d\mu_{s,N} \bigg| \notag\\
& \leq\bigg( \int |F(\Phi_{t}^{N}(u))|^{p'}\ind_{B_R}(u)d\mu_{s,N}\bigg)^{\frac{1}{p'}} \| f_{\tau}^{N}\ind_{B_R}\|_{L^{p}(d\mu_{s,N})}   \notag\\
& = \bigg( \int |F(u)|^{p'} f_{t}^{N}(u) \ind_{B_R}(u)d\mu_{s,N}  \bigg)^{\frac{1}{p'}}\| f_{\tau}^{N}\ind_{B_R}\|_{L^{p}(d\mu_{s,N})} \notag \\
&\leq \|F\|_{L^{r'}(d\mu_{s,N})} \|f_{t}^{N}\ind_{B_R}\|_{L^{\frac{r(p-1)}{p-r}}(d\mu_{s,N})}^{\frac{1}{p'}}\| f_{\tau}^{N}\ind_{B_R}\|_{L^{p}(d\mu_{s,N})}. \label{timestep1}
\end{align}
Now, \eqref{timestep} follows from \eqref{timestep1} and duality.

We now iterate \eqref{timestep} to cover general $t\in \R$. We consider positive times as the argument for negative times is the same. Given $1<r_1<\infty$, let $p>r_1$ be sufficiently large so that 
\begin{align}
\frac{p^2}{2p-1} >r_1 \label{p}
\end{align}
and choose $\tau_{R}>0$ such that
\begin{align}
\frac{r_1 (p-1)}{p-r_1}\tau_{R}=\frac{c_0}{R}, \label{tau}
\end{align}
  where $c_0$ is the constant in Lemma~\ref{LEM:expint}. For $j\in \N$, we define the sequence of times $t_j= j \tau_{R}$ and integrability exponents $r_{j}$ by 
\begin{align*}
r_{j+1}= \frac{r_{j}p}{p+r_{j}-1} \quad \text{and} \quad r_{j}\vert_{j=1}=r_1.
\end{align*}
Note that $r_{j}= \frac{r_{j+1}(p-1)}{p-r_{j+1}}$ and $1<r_{j+1}<r_j$ for all $j\in \N$. In fact, a straightforward argument by induction on $j\in \N$ shows that
\begin{align}
r_{j} = \frac{r_1(p-1)}{r_1 (p-1)-p(r_1-1)\frac{1}{(p')^{j}}} = \Big[ 1-\tfrac{p'}{r_1'} e^{-j(\log p')}\Big]^{-1} =  \Big[ 1-\tfrac{p'}{r_1'} e^{-\frac{ t_{j}}{\tau_{R}}(\log p')}\Big]^{-1}. \label{rj}
\end{align}
Thus, \eqref{timestep} implies 
\begin{align}
\| f_{t_{j+1}}^{N}\ind_{B_R} \|_{L^{r_{j+1}}(d\mu_s)} \leq \|f_{\tau}^{N}\ind_{B_R}\|_{L^{p}(d\mu_s)} \|f_{t_j}^{N}\ind_{B_R} \|_{L^{r_{j}}(d\mu_s)}^{\frac{1}{p'}}.
\end{align}
This sets up a recursive inequality for which we get
\begin{align}
\| f_{t_{j}}^{N}\ind_{B_R} \|_{L^{r_{j}}(d\mu_s)} \leq \|f_{\tau}^{N}\ind_{B_R}\|_{L^{p}(d\mu_s)}^{\sum_{k=0}^{j}(p')^{-k}} \leq \max( 1, \| f_{\tau}^{N} \ind_{B_R}\|_{L^{p}(d\mu_s)})^{j}.
\end{align}
Therefore, by Lemma~\ref{LEM:LPbound},
\begin{align}
\sup_{N\in \N}\| f_{t_{j}}^{N}\ind_{B_R} \|_{L^{r_{j}}(d\mu_s)} \le \exp\big(\wt{C}(R){\tfrac{t_j}{\tau_R}}\big), \label{densitytj}
\end{align}
where $\wt{C}(R):= c_0 R^{-1}C(R)$. 

We now obtain bounds for general $t>0$. As we need to account for a further deterioration in the integrability exponent, and in view of \eqref{rj}, we define $q:[0,\infty)\to (1,\infty)$ by 
\begin{align}
q(t) := \Big[ 1-\tfrac{p'}{r_1'} \exp(-\tfrac{2t}{\tau_{R}}(\log p'))\Big]^{-1}. \label{qt}
\end{align}
We see that $q(t)$ is monotonically decreasing, $q(0) = \frac{r_1 (p-1)}{p-r_1}<p$ by \eqref{p}, and by \eqref{tau}, 
\begin{align*}
q(t) t \leq q(0)t \leq q(0) \tau_{R} =\tfrac{c_0}{R}
\end{align*}
for all $0\leq t\leq \tau_{R}$. Therefore, we may directly apply Lemma~\ref{LEM:LPbound} to obtain
\begin{align}
\sup_{N\in \N} \|f_{t}^{N}\ind_{B_R}\|_{L^{q(t)}(d\mu_{s,N})} \leq e^{C(R)(|t|q(t))^{A}}. \label{shortime}
\end{align}
Now assume that $t>\tau_{R}$ and such that there is $j_0\in \N$ (which depends on $t$ and $\tau_R$) such that $j_0 \tau_{R} <t <(j_0 +1)\tau_{R}$. By repeating the argument which derived \eqref{timestep}, we obtain
\begin{align}
\| f_{t}^{N}\ind_{B_R}\|_{L^r (d\mu_{s,N})} \leq \|f_{t-j_0 \tau_{R}}^{N}\ind_{B_R}\|_{L^{p}(d\mu_{s,N})} \|f_{j_0 \tau_{R}}^{N} \ind_{B_R}\|_{L^{ \frac{r(p-1)}{p-r}}(d\mu_{s,N})}^{\frac{1}{p'}} \label{timestep2}
\end{align}
for any $1\leq r<p<\infty$.  We apply \eqref{timestep2} with $r=q(t)$. Since $j_0 <\frac{t}{\tau_{R}}<j_0+1$, it follows from \eqref{qt} and \eqref{rj} that
\begin{align*}
\tfrac{q(t)(p-1)}{p-q(t)}=q\big(t-\tfrac{\tau_R}{2}\big) \leq r_{j_0}.
\end{align*}
Thus, by \eqref{timestep2}, \eqref{shortime}, H\"{o}lder's inequality, and \eqref{densitytj}, we obtain 
\begin{align*}
\| f_{t}^{N}\ind_{B_R}\|_{L^{q(t)}(d\mu_{s,N})}
 & \leq e^{C(R)(\tau_{R}q(\tau_R))^{A}}
 \| f_{j_0 \tau_{R}}\ind_{B_R}\|_{L^{q(t-\frac{\tau_R}{2})}(d\mu_{s,N})}^{\frac{1}{p'}}  \\
 &\leq e^{C(R)(\tau_{R}q(0))^{A} } \| f_{j_0 \tau_{R}}\ind_{B_R}\|_{L^{r_{j_0}}(d\mu_{s,N})}^{\frac{1}{p'}} \\
 & \leq e^{C(R)(\tau_{R}q(0))^{A} } e^{\wt{C}(R)\frac{t_{j_0}}{p'\tau_{R}}} \\
 & \leq e^{\wt{C}(R)(\frac{R(p-1)t}{c_0}+1)},
\end{align*}
uniformly in $N\in \N$. Combining this with \eqref{densitytj} and \eqref{shortime} completes the proof.
\end{proof}

We now move onto the proof of Theorem~\ref{THM:QI}.

\begin{proof}[Proof of Theorem~\ref{THM:QI}]

Let $\frac{1}{2}<\s<s$. By the global well-posedness for \eqref{BOBBM} and \eqref{BON} in $H^{\s}(\T)$ \cite{Mammeri1}, and a standard argument, for each $u_0\in H^{\s}$, we have 
\begin{align}
\lim_{N\to \infty}\| \Phi_{t}(u_0)-\Phi_{t}^{N}(\Pi_{\leq N}u_0)\|_{C([-T,T];H^{\s'})}=0 \label{flowconv}
\end{align}
for every $\s'<\s$ and $T>0$.

Fix $t\in \R$ and $R>0$. Then, with $p=p(t,R)>1$ as in Proposition~\ref{PROP:LPboundlong}, \eqref{Lpbd} implies that there exists $f_{t,R}\in L^{p}(d\mu_s)$ such that, up to a subsequence, $(f_{t}^{N}\ind_{B_R})\circ \Pi_{\leq N}$ converges weakly to $f_{t,R}$ in $L^p(d\mu_s)$ as $N\to\infty$.

We now show that $f_{t,R}$ is supported on the ball $B_{R}$ in $H^{1/2}_0(\T)$.  Let $\dl>0$.
By the weak convergence of $(f_{t}^{N}\ind_{B_R})\circ \Pi_{\leq N}$, independence of $\Pi_{\leq N}u$ and $\Pi_{>N}u$ for $N\in \N$, and \eqref{decayest}, we have
\begin{align*}
\int &f_{t,R}(u)\ind_{B_{R+\dl}^{c}}(u)d\mu_{s}(u)  \\
& = \lim_{N\to \infty} \int f_{t}^{N}(\Pi_{\leq N}u) \ind_{B_R}(\Pi_{\leq N}u) \ind_{B_{R+\dl}^{c}}(u)d\mu_{s} \\
& = \lim_{N\to \infty} \int f_{t}^{N}(u_N) \ind_{B_R}(u_N) \int \ind_{B_{R+\dl}^{c}}(u_N +u_{N}^{\perp}) d\mu_{s,N}^{\perp} d\mu_{s,N} \\
& =  \lim_{N\to \infty} \int f_{t}^{N}(u_N) \ind_{B_R}(u_N)  \mu_{s,N}^{\perp}\bigg( \|u_{N}^{\perp}\|_{H^{1/2}_{0}}^2 > (R+\dl)^{2} -\|u_N\|_{H^{1/2}_{0}}^{2} \bigg) d\mu_{s,N} \\
& \leq C \limsup_{N\to \infty} \frac{1}{\dl RN^{2s-1}}  \int f_{t}^{N}(u_N) \ind_{B_R}(u_N)   d\mu_{s,N} \\
& \leq C \limsup_{N\to \infty} \frac{1}{\dl RN^{2s-1}}   =0.
\end{align*} 
Thus, 
\begin{align*}
\int f_{t,R}(u) \ind_{B_{R+\dl}^{c}}(u) d\mu_s =0
\end{align*}
for every $\dl>0$, and so by the monotone convergence theorem, we conclude that 
\begin{align}
f_{t,R} \ind_{B_{R}^{c}}=0 \quad \mu_s \,\,\text{a.e.} \label{ftRsupp}
\end{align}

Now we show that $f_{t,R}$ is the Radon-Nikodym derivative of $(\Phi_{t})_{\#}\mu_{s,R}$ with respect to $\mu_{s,R}$. Assuming this for now, we get $(\Phi_{t})_{\#}\mu_{s,R} \ll \mu_{s,R}$. By taking $R\to\infty$, this implies $(\Phi_{t})_{\#}\mu_{s} \ll \mu_{s}$. Let $F:H^{1/2}_{0}(\T)\to \R$ be continuous, bounded, and non-negative. 
By the convergence in \eqref{flowconv} and the dominated convergence theorem 
\begin{align}
\lim_{N\to \infty} \int |F(\Phi_{t}(u)) -F(\Phi^{N}_{t}(\Pi_{\leq N}u))| \ind_{B_R}(\Pi_{\leq N}u) d\mu_s = 0. \label{conv1}
\end{align}
Similarly, by \eqref{Lpbd} and the dominated convergence theorem,
\begin{align}
\lim_{N\to \infty} \int |F(u) -F(\Pi_{\leq N}u)|  f_{t}^{N}(\Pi_{\leq N}u)\ind_{B_R}(\Pi_{\leq N}u) d\mu_s = 0. \label{conv2}
\end{align}
Then, by monotone convergence, \eqref{conv1}, \eqref{conv2} and \eqref{ftRsupp}, we have
\begin{align*}
\int F(u) d(\Phi_{t})_{\#}\mu_{s,R} & = \int F(\Phi_{t}(u)) \ind_{B_R}(u) d\mu_s \\
& = \lim_{N\to \infty} \int F(\Phi_{t}(u)) \ind_{B_R}(\Pi_{\leq N}u) d\mu_s  \\
& = \lim_{N\to \infty} \int F(\Phi^{N}_{t}(\Pi_{\leq N}u))\ind_{B_R}(\Pi_{\leq N}u) d\mu_s \\
& = \lim_{N\to \infty} \int F(\Phi^{N}_{t}(\Pi_{\leq N}u))\ind_{B_R}(\Pi_{\leq N}\Phi_{t}^{N}(\Pi_{\leq N}u))  d\mu_s\\
& = \lim_{N\to \infty} \int F(\Pi_{\leq N}u) f_{t}^{N}(\Pi_{\leq N}u) \ind_{B_R}(\Pi_{\leq N}u)d\mu_{s} \\
& = \lim_{N\to \infty} \int F(u) f_{t}^{N}(\Pi_{\leq N}u) \ind_{B_R}(\Pi_{\leq N}u)d\mu_{s} \\
& = \int F(u) f_{t,R}(u) d\mu_s  \\
& = \int F(u) f_{t,R}(u) d\mu_{s,R}.
\end{align*} 
As this identity holds for all such $F$, we obtain that $d(\Phi_{t})_{\#}\mu_{s,R} = f_{t,R} d\mu_{s,R}$.
This completes the proof of Theorem~\ref{THM:QI}.
\end{proof}

\section{Proof of Lemma~\ref{LEM:expint}} \label{SEC:expint}

We need exponential integrability of $\Q_{s,N}$, uniformly in $N$ and locally on the support of $\mu_{s,N}$. Using the simplified variational formula \cite[Lemma 2.6]{FT2}, we have
\begin{align}
\begin{split}
\log \int &\exp\big( \ld |\Q_{s,N}(u)| \ind_{B_{R}}(u)\big) d\mu_{s,N} \\
& \leq  \E \bigg[  \sup_{V\in H^{s+1/2}}\Big\{ \ld |\Q_{s,N}(Y_{N}+V_{N})|\ind_{B_R}(Y_{N}+V_{N}) -\tfrac12 \|V_N\|_{H^{s+1/2}}^{2}      \Big\}  \bigg], 
\end{split}\label{BDP}
\end{align}
where $Y$ is distributed according to $\mu_{s,N}$, $Y_{N}:=\P_{\leq N}Y$, and $V_{N}:=\P_{\leq N}V$. 
For $Y_{N}+V_{N}\in B_{R}$, the triangle inequality implies 
\begin{align}
\|V_{N}\|_{H^{1/2}} \leq R^{1/2}+ \|Y_{N}\|_{H^{1/2}}. \label{triangle}
\end{align}
The goal is to show that the right-hand side of \eqref{BDP} is finite, uniformly in $N\in \N$, leveraging both the cutoff $\ind_{B_R}$ and the negative terms $\frac12 \|V_N\|_{H^{s+1/2}}^{2}$. To do this, we need to decompose the operator $\Q_{s,N}$ as in \eqref{Qdecomp} and then further decompose each of $\Q_{s,N}^{(j)}$, $j=1,2$ as in \eqref{Q1} and \eqref{Q2}
into a certain \textit{finite} number of parts $L\in \N$; namely,
\begin{align*}
|\Q_{s,N}(Y_N+V_N)| \les \sum_{j=1}^{2}\sum_{\l=1}^{L} |\Q^{(j)}_{s,N,\l}(Y_N+V_N)|
\end{align*}
for some operators $\{\Q^{(j)}_{s,N,\l}\}_{\l=1}^{L}$, $j=1,2$, and control their contributions to \eqref{BDP} separately. 

In order to use the negative terms in $\frac12 \|V_N\|_{H^{s+1/2}}^{2}$ for each of these $\{\Q_{s,N,\l}\}_{\l=1}^{L}$, we use
\begin{align*}
\text{RHS} \eqref{BDP} \leq \sum_{j=1}^{2}\sum_{\l=1}^{L} \E\Big[ \sup_{V\in H^{s+1/2}}\Big\{ \ld |\Q^{(j)}_{s,N,\l}(Y_N+V_N)|\ind_{B_R}(Y_N+V_N) -\tfrac{1}{4L} \|V_N\|_{H^{s+1/2}}^{2}   \big\} \Big].
\end{align*}
It then suffices to control each of the summands. To be concrete, it will be more than enough in the following to take $L= 100$.

In the following sections, to simplify the notation, we write $Y$ and $V$ in place of $Y_N$ and $V_N$, respectively.

\subsection{Contribution from $\Q^{(1)}_{s,N}$}
In view of Lemma~\ref{LEM:Psis}, we make the following dyadic decompositions:
\begin{align}
\Q^{(1)}_{s,N}(u) 
& = \sum_{ \substack{N_1, N_2, N_3 \in 2^{\N_0} \\  N_{1} \sim N_{2} \sim N_{3}}}\Q^{(1)}_{s,N}( \P_{N_1}u, \P_{N_2}u, \P_{N_3}u)  \label{case1}\\
& \hphantom{X}+ \sum_{ \substack{N_1, N_2, N_3 \in 2^{\N_0} \\  N_{(2)}\gg N_{(3)}}}\Q^{(1)}_{s,N}( \P_{N_1}u, \P_{N_2}u, \P_{N_3}u). \label{case2}
\end{align}

We first consider the contribution from \eqref{case2}. By symmetry, we may assume that $N_1 \geq N_2 \geq N_3$.
Then, with $u=Y+V$, we have the following unique cases to consider:
\begin{align*}
(Y_{N_1},Y_{N_2},Y_{N_3}), \,\,\,(Y_{N_1},Y_{N_2},V_{N_3}), \,\,\, (Y_{N_1},V_{N_2},Y_{N_3}), \,\,\, (Y_{N_1}, V_{N_2},V_{N_3}), \,\,\,    (V_{N_1},V_{N_2},Y_{N_3}+V_{N_3}).
\end{align*}
In these cases, by Lemma~\ref{LEM:Psis}, we always have $|\Psi_{s}(\cj{n})|\les N_{1}^{2s} N_3$.

\smallskip
\noi
\underline{\textbf{Case 1:} $(Y_{N_1},Y_{N_2},Y_{N_3})$}

\smallskip
\noi
For this case, we need to use the gain coming from random oscillations. We have
\begin{align*}
\E\big[ |\Q^{(1)}_{s,N}(Y_{N_1}, Y_{N_2}, Y_{N_3})|^{2}\big]& = \E \bigg[  \bigg|   \sum_{ n_{123}=0} \Psi_{s}(\cj{n}) \prod_{j=1}^{3} \phi_{N_j}(n_j) \ft Y(n_j)    \bigg|^{2} \bigg] \\
& = \sum_{ \substack{n_{123}=0 \\ m_{123}=0}} \Psi_{s}(\cj{n}) \Psi_{s}(\cj{m})\prod_{j=1}^{3} \frac{ \phi_{N_j}(n_j) \phi_{N_j}(m_j)  }{ |n_j|^{s+1/2}|m_j|^{s+1/2}} \E[ g_{n_1}g_{n_2}g_{n_3}\cj{g_{m_1}g_{m_2}g_{m_3}}].
\end{align*}
We use Isserlis theorem to control the expectation and exploit the independence structure.
As all the individual frequencies are non-zero, we cannot have any self-interactions of the form $n_{j}+n_{j'}=0$ for $j\neq j'$. Moreover, since $N_1, N_2 \gg N_3$, we must have that $n_3= m_3$ and thus $(n_1,n_2)=(m_1,m_2)$ or $(n_1,n_2)=(m_2,m_1)$. Both of these cases yield the same contribution, so we then have
\begin{align*}
& \les \sum_{ \substack{n_{123} =0 \\ |n_j| \sim N_j, j=1,2,3}} |\Psi_{s}(\cj{n})|^{2} (|n_1||n_2||n_3|)^{-2s-1} \les N_{1}^{4s} N_{3}^{2} N_{1}^{-4s-2}N_{3}^{-2s-1} N_{1}N_{3}   \les N_{1}^{-1}N_{3} N_{3}^{1-2s}
\end{align*}
By first summing in $N_1$ over $N_1 \gg N_3$ and using that $s>\frac 12$ to sum in $N_3$, this contribution to \eqref{BDP} is thus controlled.

\smallskip
\noi
\underline{\textbf{Case 2:} $(Y_{N_1},Y_{N_2},V_{N_3})$}

\smallskip
\noi
We write 
\begin{align}
\Q^{(1)}_{s,N}( Y_{N_1}, Y_{N_2}, V_{N_3}) = \sum_{n_3} |n_3|^{s+\frac 12} \ft V(n_3)  \phi_{N_3}(n_3)|n_3|^{-s-\frac 12} K^{(1)}_{N_1 N_2 N_3}(n_3), \label{Q1K1}
\end{align}
where 
\begin{align}
K^{(1)}_{N_1 N_2 N_3}(n_3) : =\sum_{n_1, n_2} \ind_{\{n_1+n_2+n_3=0\}} \Psi_{s}(\cj{n}) \phi_{N_1}(n_1) \phi_{N_2}(n_2) \ft Y(n_1) \ft Y(n_2). \label{K1}
\end{align}
Let $0<\dl <\min(s-\frac 12, \frac 12)$ and $\eps=\eps(s,\dl,L)>0$ such that 
\begin{align*}
2L \eps\sum_{ \substack{N_1, N_2, N_3 \in 2^{\N_0}  \\  N_1 \sim N_{2}\gg N_{3}}}N_{1}^{-2\dl} \leq \frac{1}{4L}.
\end{align*}
 Then, by Cauchy-Schwarz and Cauchy's inequality, we have
\begin{align*}
\ld |\Q^{(1)}_{s,N}( Y_{N_1}, Y_{N_2}, V_{N_3}) | & \leq \eps N_{1}^{-2\dl} \| V\|_{H^{s+\frac 12}}^{2} + \ld ^2 C_{\eps}N_{1}^{2\dl}\sum_{n_3} |\phi_{N_3}(n_3)|^2 |n_3|^{-2s-1} |K^{(1)}_{N_1 N_2 N_3}(n_3)|^{2}.
\end{align*}
It remains to control the expected value of the second term. As in Case 1, we have 
\begin{align*}
\E \bigg[  N_{1}^{2\dl} \sum_{n_3} & |\phi_{N_3}(n_3)|^2 |n_3|^{-2s-1} |K^{(1)}_{N_1 N_2 N_3}(n_3)|^{2} \bigg]\\
 & \sim N_{1}^{2\dl} N_{3}^{-2s-1} \sum_{|n_3|\sim N_3} 
 \sum_{ \substack{n_1,n_2 \\ n_1+n_2+n_3=0}} \frac{|\Psi_{s}(\cj{n}) \phi_{N_1}(n_1) \phi_{N_2}(n_2)|^2 }{(|n_1||n_2|)^{2s+1}} \\
 & \les N_{1}^{2\dl} N_{3}^{-2s-1} \cdot N_{3}N_{1} \cdot N_{1}^{-4s-2} \cdot N_{1}^{4s}N_{3}^{2} \\
 & \sim \bigg( \frac{N_3}{N_1}\bigg)^{1-2\dl} N_{3}^{-2s+1+2\dl}.
\end{align*}
This is summable, first over $N_1\gg N_3$, and then over $N_3$ since $s>\frac 12$.

\smallskip
\noi
\underline{\textbf{Case 3:} $(Y_{N_1},V_{N_2},Y_{N_3})$}

\smallskip
\noi
In this case, we proceed similarly to Case 2, writing
\begin{align*}
\Q^{(1)}_{s,N}( Y_{N_1}, V_{N_2}, Y_{N_3})  = \sum_{n_2} |n_2|^{s+\frac 12} \ft V(n_2) \phi_{N_2}(n_2)|n_2|^{-s-\frac 12} K^{(2)}_{N_1 N_2 N_3}(n_2),
\end{align*}
where 
\begin{align*}
K^{(2)}_{N_1 N_2 N_3}(n_2) : = \sum_{n_1, n_3} \ind_{\{ n_{123}=0\}} \Psi_{s}(\cj{n}) \phi_{N_1}(n_1)\phi_{N_3}(n_3) \ft Y(n_1) \ft Y(n_3).
\end{align*}
Note that $\E[g_{n_1}g_{n_3}\cj{g_{m_1}g_{m_3}}]= \dl_{n_1m_1}\dl_{n_3 m_3}$ since $|n_1|\gg |n_3|$.
Thus, we have 
\begin{align*}
\E \bigg[ N_{2}^{2\dl} \sum_{n_2} |\phi_{N_2}(n_2)|^2 |n_2|^{-2s-1} |K^{(2)}_{N_1 N_2 N_3}(n_2)|^{2} \bigg] &\sim N_2^{2\dl-2s-1} \cdot N_{1}^{4s}N_{3}^{2} \cdot N_{1}N_{3}\cdot (N_{1}N_{3})^{-2s-1} \\
& \les \bigg( \frac{N_3}{N_1}\bigg)^{1-2\dl} N_{3}^{-2s+1+2\dl}.
\end{align*}
We then sum as in Case 2.

\smallskip
\noi
\underline{\textbf{Case 4:} $(Y_{N_1},V_{N_2},V_{N_3})$}

\smallskip
\noi
A simple Cauchy-Schwarz argument, that uses only the regularity of $Y_{N_1}$ and not its Gaussian structure, suffices to control this contribution when $s<1$ but fails when $s\geq 1$. In the following, we present a unified approach that deals with all $s>\frac 12$, which does use the explicit structure of $Y_{N_1}$. 
Let $0<\dl<\frac 14 (s-\frac12)$, $\eps=\eps(s,\dl)>0$ such that 
\begin{align*}
2L \eps  \sum_{ \substack{N_1, N_2, N_3 \in 2^{\N_0}  \\  N_1 \sim N_{2}\gg N_{3}}}N_{1}^{-\frac{2\dl s}{2s-
\dl}} \leq \frac{1}{4L},
\end{align*}
and define the random variable
\begin{align*}
B_{N_1}(\o) : = \bigg( \sum_{n_1 \in \Z_{\ast}} |\phi_{N_1}(n_1)|^{2} |\ft Y(n_1)|^{2} \bigg)^{1/2}.
\end{align*}
By Cauchy-Schwarz, interpolation, \eqref{triangle}, and Young's inequality, we have 
\begin{align*}
\ld |\Q^{(1)}_{s,N}( Y_{N_1}, V_{N_2}, Y_{N_3})| & \les \ld  N_{1}^{2s} N_{3} \sum_{n_{123}=0} |\ft V_{N_1}(n_1)| |\ft Y_{N_2}(n_2)| |\ft V_{N_3}(n_3)| \\
& \les  \ld N_{1}^{2s}N_{3} \cdot N_{2}^{-s-\frac 12} N_{3}^{\frac 12} \cdot N_{3}^{-s-\frac 12 +\dl} B_{N_1}(\o)
 \|V_{N_2}\|_{H^{s+\frac 12}} \|V_{N_3}\|_{H^{s+\frac 12 -\dl}} \\
& \les \ld  N_{1}^{s-\frac 12}N_{3}^{1-s+\dl} B_{N_1}(\o) 
\|V_{N_2}\|_{H^{s+\frac 12}} \|V_{N_3}\|_{H^{s+\frac 12}}^{1-\frac {\dl}{s}} \|V_{N_2}\|_{H^{\frac 12}}^{\frac{\dl}{s}} \\
& \les \ld  N_{1}^{s-\frac 12}N_{3}^{1-s+\dl}B_{N_1}(\o)  (R^{\frac 12}+\|Y_{N_3}\|_{H^{\frac 12}})^{\frac{\dl}{s}}
\|V_N\|_{H^{s+\frac 12}}^{2-\frac{\dl}{s}} \\
& \les  \ld^{\frac{2s}{\dl}}C_{\eps,\dl} \big[N_{1}^{s-\frac 12+\dl}N_{3}^{1-s+\dl}B_{N_1}(\o)  (R^{\frac 12}+\|Y_{N_3}\|_{H^{\frac 12}})^{\frac{\dl}{s}} \big]^{\frac{2s}{\dl}} + \eps N_{1}^{-\frac{2\dl s}{2s-\dl}} \|V_N\|_{H^{s+\frac 12}}^{2}.
\end{align*}
It remains to control the first term under after taking an expectation. For any finite $p\geq 2$, Minkowski's inequality implies 
\begin{align}
\E \big[ B_{N_1}(\o)^{p}\big]^{\frac 1p} \les \bigg( \sum_{n_1 \in \Z_{\ast}} |\phi_{N_1}(n_1)|^{2} |n_1|^{-2s-1}  \E[|g_{n_1}|^{p}]^{\frac 2p}\bigg)^{\frac 12} \les_p   N_{1}^{-s}. \label{BN1bd}
\end{align}
Thus, by Cauchy-Schwarz and \eqref{BN1bd}, we have 
\begin{align*}
\E \bigg[ & \sum_{ \substack{N_1, N_2, N_3 \in 2^{\N_0}  \\  N_1 \sim N_{2}\gg N_{3}}} \big[N_{1}^{s-\frac 12+\dl}N_{3}^{1-s+\dl}B_{N_1}(\o)  (R^{\frac 12}+\|Y_{N_3}\|_{H^{\frac 12}})^{\frac{\dl}{s}} \big]^{\frac{2s}{\dl}}\bigg]\\
& =  \sum_{ \substack{N_1, N_2, N_3 \in 2^{\N_0}  \\  N_1 \sim N_{2}\gg N_{3}}}  \big(N_{1}^{s-\frac 12+\dl}N_{3}^{1-s+\dl} \big)^{\frac{2s}{\dl}} \E\Big[ B_{N_1}(\o)^{\frac{2s}{\dl}} (R^{\frac 12}+\|Y_{N_3}\|_{H^{1/2}})^{2}\Big]\\
&  \les \sum_{ \substack{N_1, N_2, N_3 \in 2^{\N_0}  \\  N_1 \sim N_{2}\gg N_{3}}}  \big(N_{1}^{s-\frac 12+\dl}N_{3}^{1-s+\dl} \big)^{\frac{2s}{\dl}} \E \big[ B_{N_1}(\o)^{\frac{4s}{\dl}}\big]^{\frac{1}{2}}  \E[ (R^{\frac 12}+\|Y_{N_3}\|_{H^{1/2}})^{4}]^{\frac{1}{2}}  \\ 
& \les_{R}  \sum_{ \substack{N_1, N_2, N_3 \in 2^{\N_0}  \\  N_1 \sim N_{2}\gg N_{3}}}  \big(N_{1}^{s-\frac 12+\dl}N_{3}^{1-s+\dl} N_{1}^{-s} \big)^{\frac{2s}{\dl}}\\ 
& \les_{R}  \sum_{ \substack{N_1, N_2, N_3 \in 2^{\N_0}  \\  N_1 \sim N_{2}\gg N_{3}}}  \Big(  \frac{N_3}{N_1} \Big)^{\frac{2s}{\dl}(\frac 12-\dl)} N_{3}^{(\frac 12-s+2\dl) \frac{2s}{\dl}}
\end{align*}
which is summable, first in $N_1\gg N_3$ and using that $s>\frac 12$ and $\dl$ is sufficiently small. 

\smallskip
\noi
\underline{\textbf{Case 5:} $(V_{N_1},V_{N_2},Y_{N_3}+V_{N_3})$}

\smallskip
\noi
By Cauchy-Schwarz, we have 
\begin{align*}
|\Q^{(1)}_{s,N}&( V_{N_1}, V_{N_2}, Y_{N_3}+V_{N_3})| \\
& \les N_{1}^{2s}N_{3}\bigg( \sum_{ \substack{n_{123}=0 \\ |n_j|\sim N_j}} |\ft V_{N_1}(n_1)|^{2} \bigg)^{\frac 12}\bigg( \sum_{ \substack{n_{123}=0 \\ |n_j|\sim N_j}} |\ft V_{N_2}(n_2)|^{2}|\mathcal{F}\{ Y_{N_3}+V_{N_3}\}(n_3)|^2 \bigg)^{\frac 12} \\
&\les N_{1}^{-1} N_{3} \|V_{N_1}\|_{H^{s+\frac 12}} \|V_{N_2}\|_{H^{s+\frac 12}} \|Y_{N_3}+V_{N_3}\|_{H^{\frac 12}}.
\end{align*}
We can handle the $N_3$ summation using $\sum_{N_3 \ll N_1} N_{1}^{-1}N_{3}\les 1$, and then the summations over $N_1$ and $N_2$ since 
\begin{align}
\sum_{N_1 \sim N_2}  \|V_{N_1}\|_{H^{s+\frac 12}} \|V_{N_2}\|_{H^{s+\frac 12}} \les \sum_{N_1} \|V_{N_1}\|_{H^{s+\frac 12}}\|\wt{\P}_{N_1} V_{N_1}\|_{H^{s+\frac 12}}\les  \|V\|_{H^{s+\frac 12}}^{2}. 
\end{align}
Thus, for $Y+V\in B_{R}$, there is a $C>0$ independent of $N\in \N$ and $R>0$, such that
\begin{align}
\sum_{ \substack{N_1, N_2, N_3 \in 2^{\N_0} \\  N_1 \sim N_{2}\gg N_{3}}} \ld |\Q^{(1)}_{s,N}( V_{N_1}, V_{N_2}, Y_{N_3}+V_{N_3})|  \leq   C \ld R^{\frac 12} \|V\|_{H^{s+\frac 12}}^{2}. \label{criticalbound}
\end{align}
Thus, 
\begin{align*}
\E \Big[ \sup_{V\in H^{s+1/2}} \Big\{  C \ld R^{\frac 12} \|V\|_{H^{s+1/2}}^{2} 
- \tfrac{1}{4L} \|V\|_{H^{s+1/2}}^{2} \Big\} \Big] \leq  0
\end{align*}
provided that we choose $\ld>0$ so that $C \ld R^{\frac 12} \leq \frac{1}{4L}$.

This completes the estimates for \eqref{case2}. 
We move onto \eqref{case1}, where all the frequencies are similar. In this case, we always have $|\Psi_{s}(\cj{n})|\les N_{1}^{2s+1}$. By symmetry, we only have three cases depending on how many $Y$ factors there are. 

\smallskip
\noi
\underline{\textbf{Case 1:} $(V_{N_1},V_{N_2},Y_{N_3}+V_{N_3})$}

\smallskip
\noi
This contribution can be handled exactly as in Case 5 when we controlled \eqref{case2}, and we obtain 
\begin{align*}
\sum_{ \substack{N_1, N_2, N_3 \in 2^{\N_0} \\  N_1 \sim N_{2}\sim N_{3}}} |\Q^{(1)}_{s,N}( V_{N_1}, V_{N_2}, Y_{N_3}+V_{N_3})| 
& \les  \sum_{ \substack{N_1, N_2, N_3 \in 2^{\N_0} \\  N_1 \sim N_{2}\sim N_{3}}} \|V_{N_1}\|_{H^{s+\frac 12}} \|V_{N_2}\|_{H^{s+\frac 12}} \|Y_{N_3}+V_{N_3}\|_{H^{\frac 12}} \\
& \les \sum_{N_1}  \|V_{N_1}\|_{H^{s+\frac 12}} \sum_{N_2 \sim N_3}  \|V_{N_2}\|_{H^{s+\frac 12}} \|Y_{N_3}+V_{N_3}\|_{H^{\frac 12}}  \\
& \les  \sum_{N_1}  \|V_{N_1}\|_{H^{s+\frac 12}} \sum_{N_2 \sim N_1}   \|V_{N_2}\|_{H^{s+\frac 12}} \|\wt{\P}_{N_2}( Y_{N_2}+V_{N_2})\|_{H^{\frac 12}}  \\
& \les R^{\frac 12} \sum_{N_1\sim N_2}\|V_{N_1}\|_{H^{s+\frac 12}}\|V_{N_2}\|_{H^{s+\frac 12}} \\
& \les R^{\frac 12} \|V\|_{H^{s+\frac 12}}^{2},
\end{align*}
which is the same bound as in \eqref{criticalbound}.

\smallskip
\noi
\underline{\textbf{Case 2:} $(Y_{N_1},V_{N_2},Y_{N_3})$}

\smallskip
\noi
We argue exactly as in Case 3 when we handled \eqref{case2}. With $K^{(2)}_{N_1 N_2N_3}$ chosen analogously to \eqref{K1} in \eqref{Q1K1}, when $N_1 \sim N_2\sim N_3$, we obtain
\begin{align*}
\E \bigg[ N_{2}^{2\dl} \sum_{n_2} |\phi_{N_2}(n_2)|^2 |n_2|^{-2s-1} |K^{(2)}_{N_1 N_2 N_3}(n_2)|^{2} \bigg] \les N_{1}^{-2s+1+2\dl}
\end{align*}
which is summable as $s>\frac 12$ and we may choose $\dl>0$ sufficiently small.

\smallskip
\noi
\underline{\textbf{Case 3:} $(Y_{N_1},Y_{N_2},Y_{N_3})$}

\smallskip
\noi
We compute
\begin{align*}
\E\big[ |\Q^{(1)}_{s,N}(Y_{N_1}, Y_{N_2}, Y_{N_3})|^{2}\big] 
& \les \sum_{ \substack{n_{123} =0 \\ |n_j| \sim N_j, j=1,2,3}} |\Psi_{s}(\cj{n})|^{2} (|n_1||n_2||n_3|)^{-2s-1}  \\
& \les N_{1}^{4s+2}\cdot N_{1}^{-6s-3}\cdot N_{1}^{2} \sim N_{1}^{-2s+1},
\end{align*}
which is summable as $s>\frac 12$.

This completes the estimation of \eqref{case1} and thus also the estimates for the main part  $\Q^{(1)}_{s,N}$. 

\subsection{Contribution from $\Q^{(2)}_{s,N}$}

We now handle the easier contribution in \eqref{BDP} coming from $\Q^{(2)}_{s,N}$.
We further expand this operator as follows: 
\begin{align*}
\Q^{(2)}_{s,N}(u)  &= \jb{ |D_{x}|^{2s}u_{N} , \H(u_N^2)} - \jb{|D_x|^{2s}u_{N}, \H(1+|D_x|)^{-1} (u_N^2)} \\
& =: \Q^{(2,A)}_{s,N}(u) +\Q^{(2,B)}_{s,N}(u).
\end{align*}
Whilst the first term $\Q^{(2,A)}_{s,N}$ is not critical in the sense that $\Q^{(1)}_{s,N}$ is, there are still $\text{high-high-low}$ interactions which require $s$-derivatives to be absorbed by the high-frequency factors, which is too much for the $Y$ factors to handle in a naive way. Thus, we still require some decompositions depending on the factors $Y$ or $V$. As a warm up to this, we dispense with the term $\Q^{(2,B)}_{s,N}$.

\begin{lemma}\label{LEM:Q2B}
For any $\frac 12 < s< 1$, it holds that 
\begin{align}
|\Q_{s,N}^{(2,B)}(u)| \les \| u_{N}\|_{H^{1/2}}^{3}, \label{Q2B1}
\end{align}
while when $s=1$, 
\begin{align}
|\Q_{1,N}^{(2,B)}(u)| \les \| u_{N}\|_{H^{1/2}}^{2} \|u_N\|_{L^{\infty}}. \label{Q2B2}
\end{align}
\end{lemma}

\begin{proof}
By rewriting and using Cauchy-Schwarz and the fractional Leibniz rule,
\begin{align*}
|\Q_{s,N}^{(2,B)}(u)| 
& = |\jb{ |D_x|^{\frac 12} u_N , \H |D_x|^{2s-\frac 12}(1+|D_x|)^{-1}(u_N^2)}|\les \|u_N\|_{H^{\frac 12}}\| u_{N}^{2}\|_{H^{2s-\frac 32}}.
\end{align*}
If $s\leq \frac 34$, then $2s-\frac 32\leq 0$ and hence \eqref{Q2B1} follows by Sobolev embedding $H^{\frac 12}(\T)\embeds L^{4}(\T)$. If $\frac 34<s<1$, then by the fractional Leibniz rule and Sobolev embedding, we have 
\begin{align*}
\| u_N^2 \|_{H^{2s-\frac 32}} \les \| |D_x|^{2s-\frac 32} u_N\|_{L^{\frac{2}{4s-3}}} \|u_N\|_{L^\frac{1}{2(1-s)}}
& \les \| u_N\|_{H^{\frac 12}}^2,
\end{align*}
which completes the proof of \eqref{Q2B1}. 

As for \eqref{Q2B2}, we have by Cauchy-Schwarz and the fractional Leibniz rule,
\begin{align*}
|\Q_{1,N}^{(2,B)}(u)|  = | \jb{|D_x|^{\frac 12}u_N , \H |D_x|^{\frac 12}(u_N^2)}| \leq \| u_N\|_{H^{\frac 12}} \|u_N^2 \|_{H^{\frac 12}}^{2} \les \|u_{N}\|_{H^{\frac 12}} \|u_N\|_{L^{\infty}},
\end{align*}
establishing \eqref{Q2B2}.
\end{proof}

\noi
Thus, it follows from Lemma~\ref{LEM:Q2B} that when $\frac 12<s<1$, for $Y_N +V_N \in B_{R}$,
\begin{align}
|\Q^{(2,B)}_{s,N}(Y_N+V_N)| \les R^{\frac{3}{2}},
\end{align}
uniformly in $N\in \N$. When $s=1$, we have instead by Young's inequality and Sobolev embedding 
\begin{align*}
\ld |\Q^{(2,B)}_{s,N}(Y_N+V_N)| \les \ld R (\|Y_N\|_{L^{\infty}} +\|V_N\|_{L^{\infty}}) \leq C \ld R\|Y_N\|_{L^{\infty}}+\ld ^2 C_{R} +\tfrac{1}{4L} \|V_{N}\|_{H^{s+1/2}}^2,
\end{align*}
where the implicit constants $C$ and $C_{L,R}$ do not depend on $N\in \N$.

We move onto controlling $\Q^{(2,A)}_{s,N}$. We decompose each of the input functions:
\begin{align}
\Q^{(2,A)}_{s,N}(u)& = 
2 \sum_{ \substack{N_1, N_2, N_3 \in 2^{\N_0} \\ N_{2}\geq N_3}} \Q^{(2,A)}_{s,N}( \P_{N_1}u, \P_{N_2}u, \P_{N_3}u) \notag \\
& = 2 \sum_{ \substack{N_1, N_2, N_3 \in 2^{\N_0} \\  N_{2} \sim N_3 \ges N_1}}\Q^{(2,A)}_{s,N}( \P_{N_1}u, \P_{N_2}u, \P_{N_3}u) \label{Q2A1} \\
& \hphantom{X}+2\sum_{ \substack{N_1, N_2, N_3 \in 2^{\N_0} \\  N_{2} \sim N_1 \gg N_3}}\Q^{(2,A)}_{s,N}( \P_{N_1}u, \P_{N_2}u, \P_{N_3}u) \label{Q2A2}.
\end{align}
Consider first the term \eqref{Q2A1}, which we split into two cases. First, we assume that at least one of the highest frequency inputs is a $V$; say the $n_2$-frequency. With $u:=Y_{N}+V_N\in B_{R}$, Bernstein and Cauchy-Schwarz inequalities imply
\begin{align*}
|\Q^{(2,A)}_{s,N}( \P_{N_1}u, \P_{N_2}u, V_{N_3})| 
& \leq \| |D_{x}|^{2s}\H \P_{N_1}u\|_{L^{\infty}} \| V_{N_2} \cdot \P_{N_3}u\|_{L^1} \\
& \les N_{1}^{2s} \big(N_{1}^{\frac 12}\|\P_{N_1}u\|_{L^2}\big) \|V_{N_2}\|_{L^2}\|\P_{N_3}u\|_{L^{2}} \\
& \les N_{1}^{2s}N_{2}^{-s-\frac 12} N_{3}^{-\frac 12}  \big(N_{1}^{\frac 12}\|\P_{N_1}u\|_{L^2}\big) \|V_{N_2}\|_{H^{s+\frac 12}}\|\P_{N_3}u\|_{H^{\frac 12}}.
\end{align*}
Note that in the second inequality we used that $\H$ commutes with $\P_{N_1}$ and is a bounded operator on $L^2$.
Then, by Young's inequality and that $s\leq 1$,
\begin{align*}
\sum_{ \substack{N_1, N_2, N_3 \in 2^{\N_0} \\  N_{2} \sim N_3 \ges N_1}} \ld |\Q^{(2,A)}_{s,N}( \P_{N_1}u, \P_{N_2}u, V_{N_3})|  & \les \sum_{ \substack{N_1, N_2, N_3 \in 2^{\N_0} \\  N_{2} \sim N_3 \ges N_1}}  \ld \frac{N_{1}^{s}}{N_2}  \big(N_{1}^{\frac 12}\|\P_{N_1}u\|_{L^2}\big) \|V_{N_2}\|_{H^{s+\frac 12}}\|\P_{N_3}u\|_{H^{\frac 12}} \\
& \les \sum_{N_2 \sim N_3} \|V_{N_2}\|_{H^{s+\frac 12}}\|\P_{N_3}u\|_{H^{\frac 12}}  \sum_{N_1 \leq N_2} \frac{N_{1}^{s}}{N_2}\big(N_{1}^{\frac 12}\|\P_{N_1}u\|_{L^2}\big)  \\
& \les \ld R \|V\|_{H^{s+\frac12}}\\
 &\leq C_{L} \ld^2 R + \tfrac{1}{4L} \|V\|_{H^{s+\frac 12}}^{2},
\end{align*}
which is acceptable.
Now consider the case when both high-frequency inputs are Gaussians $Y$. By Cauchy-Schwarz, we have 
\begin{align*}
|\Q^{(2,A)}_{s,N}( \P_{N_1}u, Y_{N_2}, Y_{N_3})|&  \les \| |D_{x}|^{2s}\P_{N_1}u\|_{L^2} \|\wt{\P}_{N_1}\P_{\neq 0} ( Y_{N_2} Y_{N_3})\|_{L^2}   \\
&\les N_{1}^{2s-\frac 12}R^{\frac 12} \|\wt{\P}_{N_1}\P_{\neq 0} ( Y_{N_2} Y_{N_3})\|_{L^2}.
\end{align*}
Now, taking an expectation, we find
\begin{align*}
\sum_{ \substack{N_1, N_2, N_3 \in 2^{\N_0} \\  N_{2} \sim N_3 \ges N_1}}  \E[ |\Q^{(2,A)}_{s,N}( \P_{N_1}u, Y_{N_2}, Y_{N_3})| \ind_{B_{R}}(u)] & \les_{R} \sum_{ \substack{N_1, N_2, N_3 \in 2^{\N_0} \\  N_{2} \sim N_3 \ges N_1}} N_{1}^{2s-\frac 12} \E[ \|\wt{\P}_{N_1}\P_{\neq 0} ( Y_{N_2} Y_{N_3})\|_{L^2}] \\
& \les_{R} \sum_{ \substack{N_1, N_2, N_3 \\  N_{2} \sim N_3 \ges N_1}} N_{1}^{2s-\frac 12} 
\bigg( \sum_{\substack{n_{123}=0 \\ |n_j|\sim N_j, j=1,2,3}} \frac{1}{(|n_2||n_3|)^{2s+1}} \bigg)^{\frac 12}\\
& \les_{R} \sum_{ \substack{N_1, N_2, N_3 \\  N_{2} \sim N_3 \ges N_1}}
 N_{1}^{2s}N_{2}^{-2s-\frac 12} \les_{R} 1,
\end{align*}
uniformly in $N\in \N$. This completes the estimation of \eqref{Q2A1}.

It remains to consider the contribution from the term \eqref{Q2A2}. 
We split into three cases depending on the number of $V$ terms appearing in the two highest frequencies (which are $N_1$ and $N_2$). 
First, assume that there are no $V$ terms. We then claim that
\begin{align}
\E[ \|\wt{\P}_{N_3}\big[ \P_{N_2}Y \cdot \H|D_x|^{2s}\P_{N_1}Y \big]\|_{L^2}] 
\les N_{3}^{\frac 12}N_{1}^{-\frac 12}  \quad \text{under} \quad N_1\sim N_2\gg N_3,
\label{Q2AYY}
\end{align}
uniformly in $N\in \N$.
Indeed, by Cauchy-Schwarz, we have 
\begin{align*}
[\text{LHS} \eqref{Q2AYY}]^{2} \les  \sum_{ \substack{n_{123}=0 \\ |n_j|\sim N_j, j=1,2,3}} \frac{|n_1|^{4s}}{|n_1|^{2s+1}|n_2|^{2s+1}} \les N_3 N_{1}^{-1}.
\end{align*}
Now, by Cauchy-Schwarz,
\begin{align*}
|\Q^{(2,A)}_{s,N}( \P_{N_1}u, \P_{N_2}u, \P_{N_3}u) | &= | \jb{ \P_{N_2}Y \cdot \H |D_x|^{2s}\P_{N_1}Y, \P_{N_3}u}|  \\
&\les N_{3}^{-\frac 12} R^{\frac 12} \| \wt{\P}_{N_3}\big[ \P_{N_2} Y \cdot \H |D_x|^{2s} \P_{N_1}Y \big]\|_{L^2} 
\end{align*}
Then, by \eqref{Q2AYY},
\begin{align*}
\sum_{ \substack{N_1, N_2, N_3 \in 2^{\N_0} \\  N_{2} \sim N_1 \gg N_3}} 2\ld \E \big[ |\Q^{(2,A)}_{s,N}( \P_{N_1}Y, \P_{N_2}Y, \P_{N_3}(Y+V)) | \ind_{B_R}(Y+V)\big]  & \les_{\ld,R} \sum_{ \substack{N_1, N_2, N_3 \in 2^{\N_0} \\  N_{2} \sim N_1 \gg N_3}} N_{1}^{-\frac 12} \les_{\ld, R} 1. 
\end{align*}

Now, consider the case when there is exactly one $V$ and one $Y$, say $Y_{N_1}$ and $V_{N_2}$.  The remaining case is dealt with in the same way since $N_1 \sim N_2$. Then, by Bernstein and Young inequalities, we have 
\begin{align*}
\ld |\Q^{(2,A)}_{s,N}( Y_{N_1}, V_{N_2}, \P_{N_3}u) | & \les \ld \| V_{N_2}\cdot \H|D_x|^{2s}Y_{N_1}\|_{L^1} (N_{3}^{\frac 12}\|\P_{N_3}u\|_{L^2}) \\
& \les \ld  N_{2}^{-s-\frac 12}N_{1}^{2s}R \|Y_{N_1}\|_{L^2} \|V_{N_2}\|_{H^{s+\frac 12}} \\
& \leq C_{\eps} \ld^2  N_{1}^{2s-1+\dl} R^{2} \|Y_{N_1}\|_{L^2}^{2} + \eps N_{1}^{-\dl}\|V\|_{H^{s+\frac 12}}^{2},
\end{align*}
where $0<\dl<1$ and $\eps=\eps(s,\dl)>0$ such that 
\begin{align*}
2L \eps \sum_{ \substack{N_1, N_2, N_3 \in 2^{\N_0} \\  N_{2} \sim N_1 \gg N_3}}N_{1}^{-\dl} \leq \frac{1}{4L}.
\end{align*}
For the first term, we then get 
\begin{align*}
N_{1}^{2s-1+\dl} \E[ \|Y_{N_1}\|_{L^2}^{2} ] \les N_{1}^{-1+\dl}
\end{align*}
which is summable.

Finally, we have the case where both high frequency factors are $V$ terms. Similar to the previous case, we have
\begin{align*}
|\Q^{(2,A)}_{s,N}( V_{N_1}, V_{N_2}, \P_{N_3}u) | & \les \| V_{N_1}\cdot \H|D_x|^{2s}V_{N_2}\|_{L^1} (N_{3}^{\frac 12}\|\P_{N_3}u\|_{L^2})  \les RN_1^{s-\frac 12} \|\P_{N_1}V\|_{L^2} \|V\|_{H^{s+\frac 12}}.
\end{align*}
If $s<1$, then $s-\frac 12 <\frac 12$ and we have
\begin{align*}
|\Q^{(2,A)}_{s,N}( V_{N_1}, V_{N_2}, \P_{N_3}u) |  \les N_{1}^{-s+1} R^{2} \|V\|_{H^{s+\frac 12}}.
\end{align*}
By repeating the Cauchy-Schwarz argument from the previous case, we control the contribution of this term. 
If $s=1$, by \eqref{triangle} we have 
\begin{align*}
\ld |\Q^{(2,A)}_{s,N}( V_{N_1}, V_{N_2}, \P_{N_3}u) | 
& \les \ld R^{\frac 12} N_1^{2} \| V_{N_1}\|_{L^2} \|V_{N_2}\|_{L^2} \\
& \les \ld R^{\frac 12}N_{1}^{-\frac 14} \|V\|_{H^{\frac 12}}^{\frac 34} \|V\|_{H^{\frac{3}{2}}}^{\frac 54}\\
& \les \ld R^{\frac 12} N_1^{-\frac14}(R+\|Y\|_{H^{\frac 12}})^{\frac 38} \|V\|_{H^{\frac 32}}^{\frac 54} \\
& \leq C_{\eps}\ld^{\frac 85} R^{\frac{4}{5}} N_{1}^{-\frac{1}{5}} (R+\|Y\|_{H^{\frac 12}})^{\frac 35} + \eps N_{1}^{-\frac 13} \|V\|_{H^{\frac 32}}^{2}.
\end{align*}
for any $\eps>0$. Thus, the contribution from this term is controlled after performing the dyadic summations and choosing $\eps>0$ sufficiently small depending on $L$.

This completes the estimation of $\Q^{(2,A)}_{s,N}$ and thus the proof of Lemma~\ref{LEM:expint}.

\subsection{Proof of Proposition~\ref{PROP:s12}}\label{SEC:sing}
We revisit some of the computations we made in the previous sections to show the divergent lower bounds.  
Let $0<s\leq \frac 12$, $N$ fixed to be chosen large enough later, and suppose $t\neq 0$.
By the decomposition \eqref{Qdecomp} and Cauchy's inequality, we have
\begin{align}
\text{QI}_{s,N} \geq (1-\eps) \E \bigg[ \bigg| \int_{0}^{t} Q^{(1)}_{s,N}( S(-t') u_0^{\o})dt'\bigg|^{2}\bigg] -C_{\eps}  \E \bigg[ \bigg| \int_{0}^{t} Q^{(2)}_{s,N}( S(-t') u_0^{\o})dt'\bigg|^{2}\bigg], \label{QIfail2}
\end{align}
for any $\eps>0$.
Now
\begin{align*}
\int_{0}^{t} Q^{(2)}_{s,N}( S(-t') u_0^{\o})dt'= \sum_{ \substack{ n_{123}=0 \\ 0<|n_j|\leq N, \, j=1,2,3}} \frac{|n_1|^{2s}n_1}{1+|n_1|}  \frac{e^{it\Phi(\cj{n})}-1}{\Phi(\cj{n})} \prod_{j=1}^{3} \frac{g_{n_j}}{|n_j|^{s+1/2}},
\end{align*}
where 
\begin{align}
\Phi(\cj{n}):= \Phi(n_1,n_2,n_3) : = \frac{n_1}{1+|n_1|} +\frac{n_2}{1+|n_2|} +\frac{n_3}{1+|n_3|}. \label{Phase}
\end{align}
By the Isserlis' theorem, we have 
\begin{align}
\begin{split}
\E \bigg[ \bigg| \int_{0}^{t} Q^{(2)}_{s,N}( S(-t') u_0^{\o})dt'\bigg|^{2}\bigg]  & \les  t^{2} \sum_{n_{123}=0} \frac{|n_1|^{4s}}{(|n_1||n_2||n_3|)^{2s+1}} +  t^{2} \sum_{n_{123}=0} \frac{|n_1|^{2s} |n_{2}|^{2s}}{(|n_1||n_2||n_3|)^{2s+1}}\\ 
& \les t^{2}.
\end{split} \label{QIfail3}
\end{align}
Next, we also have 
\begin{align}
\E \bigg[ \bigg| \int_{0}^{t} Q^{(1)}_{s,N}( S(-t') u_0^{\o})dt'\bigg|^{2}\bigg]  &\sim  \sum_{ \substack{ n_{123}=0 \\ 0<|n_j|\leq N, \, j=1,2,3}}  \frac{| \Psi_{s}(\cj{n})|^{2}}{(|n_1||n_2||n_3|)^{2s+1}} \bigg| \frac{e^{it\Phi(\cj{n})}-1}{\Phi(\cj{n})} \bigg|^{2}  \notag \\
& \ges  t^2 \sum_{ \substack{ n_{123}=0 \\ 0<|n_j|\leq N, \, j=1,2,3}}  \frac{| \Psi_{s}(\cj{n})|^{2}}{(|n_1||n_2||n_3|)^{2s+1}},  \label{lowerbd}
\end{align}
where we used that $|\Phi(\cj{n})| \les 1$ for all $(n_1,n_2,n_3)$ and the assumption $|t|\les 1$.

We now split into two cases. First, suppose that $0<s<\frac 12$. Let $1\leq A \ll N$ to be chosen later.
Consider the subset of frequency interactions where $n_1\in \frac 12 N + [\frac 34 A, A]$, $n_{2}\in -\frac 12 N + [0,\frac{1}{4} A]$. Then, $n_3\in [-\frac 54 A, -\frac{3}{4}A]$ and by the mean value theorem,
\begin{align}
|\Psi_{s}(\cj{n})| \ges N^{2s} A  -A^{2s+1} \ges A N^{2s}. \label{phaselower}
\end{align}
Thus,
\begin{align*}
\eqref{lowerbd} &\ges t^{2}  A^{1-2s}N^{-2} \big| \{ (n_1,n_2,n_3) \in \Z_{\ast} \, : \, n_{123}=0, \,\,  n_1\in \tfrac 12 N + [\tfrac 34 A, A] ,  \\ 
&  \hphantom{XXXXXXXXXXXXXXXXXXXXXX} n_{2}\in -\tfrac 12 N + [0,\tfrac{1}{4} A], \, n_3 \in [-\tfrac 54 A, -\tfrac{3}{4}A]\} \big| \\
& \sim t^{2} A^{3-2s}N^{-2}.
\end{align*}
Choose $A=\dl N$ for any $0<\dl\ll \frac 12$ so that \eqref{phaselower} holds. Then, we choose $N$ large so that $N\gg \wt{C} C_{\eps}$, where $\wt{C}$ is the implicit constant in \eqref{QIfail3} at which point by \eqref{QIfail2} we obtain \eqref{QIfail}. 

Now assume that $s=\frac 12$. On the subset 
\begin{align*}
\G:= \{ (n_1,n_2,n_3)\in \Z_{\ast} \, : \, n_{123}=0,\,\, |n_j| \leq N, \,\, j=1,2,3, \,\, n_2,n_3<0<n_1, |n_1|\sim |n_2|\gg |n_3|\},
\end{align*}
we have 
\begin{align*}
\Psi_{\frac 12}(\cj{n})=n_1^{2}-n_2^2 -n_3^2= -n_3(n_1-n_2+n_3)= |n_3|(|n_1|+|n_2|-|n_3|).
\end{align*}
Thus, 
\begin{align*}
\eqref{lowerbd} \ges t^{2} \sum_{\G}  \frac{| \Psi_{\frac 12}(\cj{n})|^{2}}{(n_1 n_2 n_3)^{2}} \ges t^{2} \sum_{\G} \frac{1}{n_2^{2}}-t^{2} \sum_{\G} \frac{1}{n_1^2 n_2^2} \sim t^{2} (\log N)^2.
\end{align*}
This completes the $s=\frac 12$ case of \eqref{QIfail} and thus the proof of proof of Proposition~\ref{PROP:s12}.
\begin{ackno}\rm 
The author was supported by the European Research Council (grant no.~864138 ``SingStochDispDyn").
\end{ackno}


\end{document}